\documentclass[12pt]{article}
\usepackage[utf8]{inputenc}
\usepackage{amsfonts,amssymb,mathrsfs,amsmath,amscd,comment}

\numberwithin{equation}{section}

\newtheorem{theorem}{Theorem}[section]

\newtheorem{definition}{Definition}[section]

\newtheorem{remark}{Remark}[section]
\newtheorem{proof}{Proof}

\newcommand{\ep}{\varepsilon}            

\newcommand{\oo}{\overline{o}}
\newcommand{\ept}{{\ep_T}}

\makeatletter
\def\@fnsymbol#1{\ensuremath{\text{\number#1}}}
\makeatother

\begin{document}

\title{On large deviation principles for general random processes%
\footnote{Research supported by the Fundamental Scientific Research Program SB RAS, grant FWNF-2026-0030.}}

\author{Borovkov~A.~A.\!\footnote{Sobolev Institute of Mathematics, Siberian Branch of the Russian Academy of Sciences, 4 Akademika Koptyuga Prospekt, 630090 Novosibirsk,
Russian Federation; e-mail:  borovkov@math.nsc.ru.}\ {} and 
Borovkov~K.~A.\!\footnote{School of Mathematics and Statistics, The University of Melbourne, Melbourne, Australia; e-mail: borovkov@unimelb.edu.au.}}

\markboth{Borovkov~A.~A., Borovkov~K.~A.}{On large deviation principles}

\date{}

\maketitle

\begin{abstract} 
Let $Z=\{Z(t): t\in \mathbb R\}$ be a stochastic process with trajectories in  space $\mathbb D (\mathbb R)$. It is assumed that there exists an essentially smooth function $A:\mathbb R\to (-\infty, \infty] $ such that, for all
$\alpha \in \mathbb   R, $ $ \mu\in \mbox{\rm dom}\, A$, one has
\begin{equation*}
\frac1{T} \ln  {\mathbf E} \big( e^{\mu (Z(T)-\alpha T)} \big|Z(s), \ s\le 0 \big)
= A(\mu) +o(1) 
\end{equation*}
uniformly on the event $C(T):=\{|Z(0)/T - \alpha |< \eta_T \} $, where $ \eta_T \to 0$ as $T\to\infty.$ Under this condition, a uniform conditional local large deviation principle (l.l.d.p.) is established: for any fixed $\alpha, \beta\in \mathbb R$ and a positive function $\eta_T=o(1)$, for $\varepsilon_T \to 0$ sufficiently slowly as $T\to\infty,$ one has
\begin{equation*}
\lim_{T\to\infty}\frac1T \ln {\mathbf P} \big( {Z(T)}/T-\alpha 
 \in (\beta-\varepsilon_T, \beta +\varepsilon_T) \big| Z(s), \ s\le 0\big)
= - D(\beta ) 
\end{equation*}
uniformly on $C(T)$, where $D$ is the Legendre transform of the function $A$. This result is used to establish a conditional l.l.d.p.\ for the finite-dimen\-sional distributions of the process $ \{ z_T(s) = Z(sT)/T: s\in [0,1]\}$. Under additional conditions on the magnitude of oscillations of the trajectories $z_T$, a functional l.l.d.p.\ is obtained for the asymptotics of $\ln {\mathbf P} (z_T\in (f)_{\varepsilon_T})$ as $T\to\infty$, where $f\in \mathbb D(0,1),$ $(f)_\varepsilon$ is the $\varepsilon$-neighborhood of $f$ in the space $ \mathbb D(0,1)$ with respect to the uniform metric, and $\varepsilon_T \to 0$ sufficiently slowly. The obtained results can be extended to a more general triangular array scheme where the process itself $Z=Z^{(T)}$ also depends on the parameter $T$.

\medskip 

\smallskip


\end{abstract}

\section{Introduction}
\label{s1}

A systematic study of large deviation probabilities for stochastic processes (primarily for the classical scheme of sums of independent identically distribu\-ted random variables) was initiated in the seminal paper~\cite{1}. Since that time, an extensive literature has emerged devoted to the theory of large deviations, both in the classical setting and in more general frameworks. Surveys of the development of this theory and its main results can be found in~\cite{2}--\cite{7}. One of the most common ways to characterize the behavior of large deviation probabilities is the {\em large deviation principle\/} (l.d.p.; this term appears to have been first used  in~\cite{8}).

Let  $Z=\{Z(t)\colon  t\in (-\infty, \infty)\} $  be a general univariate random process with trajectories in space  $\mathbb{D} (-\infty,\infty)$ of functions with no discontinuities of the second kind, that we will assume to be right-continuous for the definiteness sake. When studying asymptotic laws for   $\{Z(t)\colon t\in [ 0,T]\} $ as  $T\to\infty,$ one most often uses the construction of the process~$Z$, provided that it is known. In this note, we will not assume known the nature of the process~$Z$ but, rather, will only assume that some asymptotic law holds true for the values of  $Z(T)$ as $T\to\infty$. Since asymptotic laws  always involve some tolerance, we will be dealing not with a particular process~$Z$ but with  \textit{classes\/}  of asymptotically equivalent (in certain sense) processes. We will state broad conditions that guarantee that the l.d.p.\ holds for classes of processes~$Z$ that can be obtained, in particular, by adding ``noise'' to known original processes. In that case, the results we obtain here can be used to study conditions on the noise that would not invalidate the asymptotic laws for the original processes.
 
For well-studied random processes, such as compound renewal processes (c.r.p.'s)  that are ``driven'' by a sequence of random vectors, the key condition for the l.d.p.\ to hold is the moment Cram\'er condition on these ``driving'' vectors.  For a broader class of processes that we are dealing with here, such conditions can become meaningless. Then one can consider a somewhat different approach to solving the problem for the c.r.p.~$Z$ and assume that the Cram\'er condition is met for the process~$Z$ itself, i.e.\ assume the existence of the so-called \textit{fundamental function}
\begin{equation}
\label{eq1.1}
A (\mu ) : = \lim_{T\to\infty} \frac1T \ln \mathbf{E} e^{\mu Z(T)} \in (-\infty, \infty], \qquad \mu\in \mathbb{R}.
\end{equation}

It turns out that the existence of such function that is convex and possesses a number of further properties is necessary and sufficient for the l.d.p.\ for the increments of the c.r.p.\ $Z(T)$ (see~\cite{9}). The latter means that, for a wide class of Borel sets~$B$, there exists the limit 
$$
\lim_{T\to\infty} \frac1T \ln \mathbf{P} \biggl( \frac{ Z(T)}T\in B\biggr) = - D (B),
$$
where  $D (B):= \inf_{\alpha \in B} D (\alpha)$ and $D(\alpha) $ is the rate function for the c.r.p.\ that is equal to the Legendre transform of $A(\mu)$:
\begin{equation}
\label{eq1.2}
D (\alpha) =\mathcal L_{A} (\alpha) := \sup_\mu (\alpha \mu - A (\mu)), \qquad \alpha \in \mathbb{R}.
\end{equation}

Similarly, for general processes $Z$ one can use the moment Cram\'er condi\-tion for the distribution of the process~$Z$  \textit{itself},  i.e.\ assume that there exists the fundamental function  $A (\mu)$ defined in~\eqref{eq1.1} that is also convex. In that case, we will be dealing with the triangular array scheme, as the distributions of the variables 
$$
z(T) := \frac{Z(T)}{T}
$$
depend on the parameter~$T$.

{According to the G\"artner--Ellis theorem (\cite{10}, \cite{11}; see also  \cite[Section~2.3]{5})}, 
for the  l.d.p.\ to hold for the random variables  $ Z(T) $, where~$Z$ is a general random process, it is sufficient, as it was the case for the c.r.p., that there exists the fundamental function~\eqref{eq1.1}, the function  $A\colon \mathbb{R}\to (-\infty, \infty] $ being ``essentially smooth'', which means that the following three conditions are met:
\begin{itemize}
\item[1)] the interior  $(\operatorname{dom}A )$ of its effective domain 
$$
\operatorname{dom}A := \{\mu\in \mathbb{R} \colon A (\mu ) <\infty\}
$$
is non-empty;


\item[2)] $A$ is differentiable in~$(\operatorname{dom}A )$;

\item[3)] one has $\lim_{k\to\infty} |A'(\mu_k)|=\infty$ for any sequence  $\{\mu_k\}_{k\ge 1}\subset (\operatorname{dom}A )$ such that there exists the limit  $\lim_{k\to\infty} \mu_k =\mu \in \partial\, \mbox{dom}\, A$ (see e.g.\ \cite[\S\,26]{12} or \cite[Definition~2.3.5]{5}; note that if the domain  $\operatorname{dom}A$ is unbounded in some direction then the respective ``point'' $\pm \infty$ is not counted as part of $\partial\, \operatorname{dom}A$);
\end{itemize}
and, moreover, that  $0\in (\operatorname{dom}A )$.

Under these conditions, the rate function~$D $ in~l.d.p.\ equals  $\mathcal L_{A} (\alpha)$ and, furthermore, the ``inverse'' relation holds true as well: 
$$
A(\mu) =\mathcal L_{D} (\mu) := \sup_\alpha ( \mu \alpha -{D}(\alpha)), \qquad \mu \in \mathbb{R}.
$$

{There is also an assertion that is converse to the G\"artner--Ellis theorem that holds true under additional integrability conditions   (Varadhan's lemma; see e.g.\ \cite[Theorem~4.3.1]{5}). Namely, if the l.d.p.\  holds true for   $Z(T)$ with a rate function~$D$ that satisfies condition~\eqref{eq2.3} below and, for some   $\mu\in \mathbb{R}$ one has 
\begin{equation}
\label{eq1.3}
\lim_{M\to\infty} \limsup_{T\to\infty} \frac1T \ln \mathbf{E} \biggl(e^{\mu Z(T)}; \frac{\mu Z(T)}{T}\ge M \biggr) =-\infty,
\end{equation}
then, for that~$\mu$, there exists the limit~\eqref{eq1.1} which is equal to $\mathcal L_D (\mu)$. Moreover, if condition~\eqref{eq1.3} is satisfied with $\mu Z(T)$ replaced by  $T\phi (Z(T)/T)$, where  $\phi\colon\mathbb{R}\to\mathbb{R}$ is an arbitrary continuous function, then there will also exist the limit 
$$
\lim_{T\to\infty} \frac1T \ln \mathbf{E} e^{T \phi (Z(T)/T) } =\sup_{\alpha}(\phi (\alpha)- D(\alpha)).
$$

The main goal of the present note is the derivation of a functional l.d.p.\ for the increments of a general process~$Z$. From a formal viewpoint, an assertion of such kind might  be obtained as a consequence of the general version of the G\"artner--Ellis theorem in topological spaces (see e.g.~\cite[Section~4.5]{5}). However, verifying the condition of the existence of the fundamental functi\-onal possessing the required properties in the corresponding dual space appears to be hardly a feasible task. As an acceptable alternative, we establish the functional l.d.p.\ for the increments of the processes~$Z$ under broad assumptions on the existence of a common fundamental function for its increments. 
 
The papers consists of the following parts. In Section~\ref{s2} we present the definitions of the key concepts related to the l.d.p.\ and establish a uniform version of the conditional G\"artner--Ellis theorem for the increments of the general process~$Z$  under the assumption of their rough asymptotic indepen\-dence. These results are used in Section~\ref{s3} to prove the so-called  \textit{local} l.d.p.\ (l.l.d.p.) for the trajectories of the process $z_T(s) := Z(sT)/T$ in the functional space $\mathbb{D}:= \mathbb{D} [0,1]$ of functions without discontinuities of the second kind endowed with the uniform metric.  

From the proofs presented below one can see that it is not hard to extend all the results obtained in the paper (with obvious changes in the assertions' statements) to the more general case where the process $Z=Z^{(T)}$ itself also depends on the parameter~$T$ and can be multivariate. We restricted ourselves in this paper to considering the simplest case to simplify the exposition.

\section{Univariate large deviation principles}
\label{s2}

First we will give a formal definition for a concept that is important in the context of this paper. 

\begin{definition}
\label{d2.1}
{\rm
We will say that a statement~$\mathcal S$  whose formulation involves a function $\varepsilon=\varepsilon_T >0$ {\em is satisfied for $\ept $ that tends  to zero slowly enough  as\/} $T\to\infty $ (and will denote this property of the function~$\ep$ by writing  ``$\varepsilon =\oo (1)$ as $T\to\infty $'') if
\begin{enumerate}
\item[(i)] there exists a positive function  $\widetilde\varepsilon = \widetilde\varepsilon_T \to 0$ as $T\to\infty  $ such that   $\mathcal S$ holds true with $\varepsilon=\widetilde\varepsilon,$ and
\item[(ii)] the statement~$\mathcal S$ holds true for all functions
$\varepsilon_T \to 0 $ such that
$\varepsilon_T\ge   c \widetilde\ep_T $ for some constant  $c>0.$
\end{enumerate}
}
\end{definition}

\begin{remark}
\label{r2.1}
{\rm 
If the relation $\varepsilon_T \ge c \widetilde \varepsilon_T $ does not hold (for instance, $\varepsilon_T \ll \widetilde \varepsilon_T $ as $T\to\infty $), then assertion $\mathcal S$ for such a function~$\varepsilon_T $ does not need to hold. If, for example, $z(T)=\frac1T\sum_{1\le k\le T} \xi_k$, $T>0$, where $\{\xi_k\}_{k\ge 1}$ is a sequence of independent identically distributed random variables with $\mathbf{E} \xi_1=0$, $\mathbf{E} \xi_1^2<\infty$, then the relation  $\mathbf{P}(|z(T)|<\varepsilon)\to 1$ as $T\to\infty$ will hold true not only for any fixed  $\varepsilon>0$ but also for $\varepsilon=\varepsilon_T \to 0$ as $T\to\infty$ provided that $\varepsilon\gg T ^{-1/2}$, i.e.\ when $\varepsilon$ vanishes slowly enough  (slower than $ T ^{-1/2}$). In this case, in Definition~\ref{d2.1} one can take, say, $\widetilde\varepsilon_T :=T ^{-1/3}$.
}
\end{remark}

\subsection{Large deviation principles. Their equivalence}
\label{ss2.1}

Recall the definition of the  l.d.p.\ for a random process.

\begin{definition}
\label{d2.2}
{\rm 
We will say that the process $Z$ satisfies the l.d.p.\ if here exists a lower semi-continuous function  
$D\colon \mathbb{R} \to [0,\infty]$, $D\not\equiv \infty,$  such that, for any Borel set $B\subset \mathbb{R}$ the following inequalities hold true for the ratio $z(T)=Z(T)/T$:
\begin{align}
\liminf_{T\to\infty}\frac1{T}\ln\mathbf{P} (z(T)\in B) &\ge - D((B)),
\notag
\\
 \limsup_{T\to\infty}\frac1{T}\ln\mathbf{P} (z(T)\in B) &\le - D([B]),
\label{eq2.1}
\end{align}
where 
$$
D(B):=\inf_{\alpha \in B}D(\alpha),
$$
$(B)$ is the interior of the set $B$ and~$[B]$ is its closure.
}
\end{definition}

The function  $D$ is called the rate function for the process~$Z$. The probabilistic meaning of the deviation function is best understood from the already mentioned   l.l.d.p.\ that was introduced and studied in~\cite{13}, \cite{4}. Denote by  $(\alpha)_\varepsilon:=(\alpha -\varepsilon, \alpha +\varepsilon)$ the  $\varepsilon$-neighborhood of the point~$\alpha$.

\begin{definition}
\label{d2.3}
{\rm
We will say that the process $Z$ satisfies the~l.l.d.p. 
if there exists a function  $D\colon\mathbb{R} \to [0,\infty]$, $D\not\equiv \infty,$ such that, for any $\alpha \in \mathbb{R} $,  one has  
\begin{align}
\lim_{\varepsilon\searrow 0}\liminf_{T\to\infty}\frac1{T}\ln\mathbf{P} (z(T)\in (\alpha)_\varepsilon) & =- D(\alpha),
\notag
 \\
\lim_{\varepsilon\searrow 0}\limsup_{T\to\infty}\frac1{T}\ln\mathbf{P} (z(T)\in (\alpha)_\varepsilon) & = - D(\alpha).
\label{eq2.2}
\end{align}
}
\end{definition}

It was shown in~\cite{4} that if the process  $Z$ satisfies the l.l.d.p.\ then the function~$D$ in Definition~\ref{d2.3} is automatically lower semi-continuous. Moreover, it was also proved there that the l.l.d.p.\  is equivalent to the following: if $\varepsilon=\varepsilon_T =\overline{o}(1)$ as $T\to\infty$ then 
$$
\lim_{T\to\infty}\frac1{T}\ln\mathbf{P} (z(T)\in (\alpha)_{\varepsilon_T}) = - D(\alpha), \qquad \alpha\in \mathbb{R}.
$$

For the classical scheme of a random walk with independent identically distributed jumps, the l.d.p.\ and l.l.d.p.\ both hold under the moment Cram\'er condition for the jump's distribution  (see e.g.~\cite{2}).

The next assertion for general processes~$Z$ follows from Theorem~1.1 in~\cite{4}
under the assumption the the rate function is ``good'', i.e.\ when
\begin{equation}
\label{eq2.3}
\text{for any } v>0\text{ the set } \{\alpha\colon D(\alpha)\le v\} \text{ is compact}.
\end{equation}

\begin{theorem}
\label{t2.1}
In the general case, the l.d.p.\ always implies the l.l.d.p. Con\-versely, let the following condition be met: for any $N>0$ there exists a $v>0$ such that 
\begin{equation}
\label{eq2.4}
\limsup_{T\to\infty} \frac1T \ln \mathbf{P} (|z(T)|>v) \le -N.
\end{equation}
Then the  l.l.d.p.\ implies the l.d.p.
\end{theorem}

Thus, the l.d.p.\ and l.l.d.p.\ are equivalent under conditions~\eqref{eq2.3} and~\eqref{eq2.4}.

\subsection{Conditional large deviation principle for increments}
\label{ss2.2}

Denote by $\{\mathcal{F}_t:=\sigma (Z(u)\colon u\le t)\colon t\in (-\infty, \infty)\}$ the filtration generated by the process~$Z$ and introduce the following uniformity concept. 

Suppose we are given a function $\{ \pi_t\colon t\ge 0\}$ and a family of events  $\{Q_{s_1, s_2}(t), C_{s_1} (t)\, $:\allowbreak $\  0\le s_1 < s_2, \, t\ge 0\}$ 
such that $Q_{s_1, s_2}(t)\in \mathcal{F}_{s_2 t}$, $C_{s_1} (t)\in \mathcal{F}_{s_1t}$. The parameter  $t$ will act as a scale factor on the time axis, whereas the values of $s_i$ will specify the locations of time epochs in the respective scale. 

Assume that, for any fixed $0\le s_1 < s_2,$ there exists a non-random function  $ \delta_{s_1, s_2}(T)\to 0$ as $T\to \infty$ such that  $\mathbf{P} (Q_{s_1, s_2}(T) \, | \, \mathcal{F}_{s_1 T}) =\pi_T(1+\Theta_{s_1, s_2}(T))$ (or $\mathbf{P} (Q_{s_1, s_2}(T) \, | \, \mathcal{F}_{sT}) =\pi_T +\Theta_{s_1, s_2}(T) $), where, for the random variables~$\Theta_{s_1, s_2}(T),$ one has  $C_{s_1}(T)\subseteq \{|\Theta_{s_1, s_2}(T)|\le \delta_{s_1, s_2}(T)\}$. In this case we will say that the relation  
\begin{align*}
\mathbf{P} (Q_{s_1, s_2}(T) \, | \, \mathcal{F}_{s_1T }) &= \pi_T (1+o(1))
\\
(\text{respectively, } \ \mathbf{P} (Q_{s_1, s_2}(T) \, | \, \mathcal{F}_{s_1T})&= \pi_T + o(1))
\end{align*}
holds  uniformly over the elementary events~$\omega\in C_{s_1}(T)$ as $T\to\infty$. The uniformity of inequalities or limiting relations for conditional probabilities or conditional expectations will be understood in a similar sense. In case of infinite limits, uniformity will be understood in the natural sense (the values are uniformly large). 

We will assume that the process $Z$ satisfies the below condition of the existence of a common fundamental function for its increments. 

\smallskip 
 
[$\mathbf{A}$].~{\it There exists an essentially smooth convex function $A\colon\mathbb{R}\to (-\infty, \infty]$ such that, for any fixed $0\le s_1 <s_2 <\infty$, $\alpha \in \mathbb{R} $, $\mu \in \operatorname{dom}A $ and a positive function $\eta=\eta_T \to 0$, $T\to\infty$, the following holds true: uniformly over elementary events}
$$
\omega \in C_{s_1}( T) := \biggl\{\biggl|\frac{Z(s_1T)}{T} - \alpha\biggr| < \eta_T \biggr\}
$$
\textit{one has}
\begin{gather*}
\frac1{T_2 - T_1} \ln \mathbf{E} \bigl( e^{\mu (Z(T_2)- \alpha T)} \bigm| \mathcal{F}_{T_1} \bigr) = A(\mu) +o(1)
\\
\textit{with } \  T_i:= s_i T, \ \ i=1,2,\ \text{ as }\ T\to\infty.
\end{gather*}

\begin{remark}
\label{r2.2}
{\rm Condition [$\mathbf{A}$]  is quite broad. It is met, for instance, in the case of the process   $Z^* (t) = Z(t) + Y(t)$, where~$Z$ is a  c.r.p.\ for which the elements of the ``driving sequence'' satisfy the moment Cram\'er condition (see e.g.\ \cite[Ch.~3]{9}), and~$Y$ is a ``noise process'' that is independent of~$Z$ and has the uniform property   $\mathbf{E} (e^{\mu Y(T)} \mid \mathcal{F}_{T_1}^Y )= e^{o(T)}$  
as $T\to\infty$, where $\{\mathcal{F}_t^Y\}$ is the filtration generated by the process~$Y$ (cf.\ e.g.\ \cite[Lemma~1\,(i)]{14}).
}
\end{remark}

\begin{theorem}
\label{t2.2}
If condition $[\mathbf{A}]$ is met then the uniform conditional l.l.d.p.\ with the rate function  $ D(\alpha) :=\mathcal L_{ A} (\alpha)$, $\alpha\in \mathbb{R},$ holds for the increments of the process~$Z\!\!:$ for any fixed $0\le s_1 <s_2 <\infty$, $\alpha, \beta\in \mathbb{R}$ and positive function  $\eta=\eta_T=o(1),$ for  $\varepsilon_T =\overline{o}(1)$ uniformly over the elementary events  $\omega \in C_{s_1}(T) $ one has 
$$
\lim_{T\to\infty}\frac1T \ln \mathbf{P} \biggl(\frac{Z(T_2)-Z(T_1)}{ T_2-T_1} \in (\beta)_{\varepsilon_T } \biggm| \mathcal{F}_{T_1}\biggr) = - (s_2 - s_1 ) D(\beta).
$$
\end{theorem}

\begin{remark}
\label{r2.3}
{\rm 
Unlike the G\"artner--Ellis theorem's conditions, the conditions of Theorem~\ref{t2.2} do not include the requirement that  $0\in (\operatorname{dom}A)$. This difference is due to the fact that above requirement was needed in the former theorem to guarantee the validity of the upper bound~\eqref{eq2.1} for  \textit{unbounded} sets~$B$ (cf.~\cite[Corollary~6.1.6]{5} and \cite[\S\,2, Remarks~1 and 2]{2}), whereas Theorem~\ref{t2.2} establishes the validity of the  \textit{local\/} l.d.p.
}
\end{remark}

\begin{proof}\!\!~{\bf of Theorem~\ref{t2.2}.} 
{\rm 
Assume for simplicity that $s_1= \alpha =0$, $s_2 = 1 $ (the proof in the general case is done in exactly the same way)  and set
$$
\zeta_T := \frac{Z(T) - Z(0)}{T}.
$$

(i)~\textit{Upper bound}. Denote by  $ \mathbf{I}(A)$ the indicator of the event~$A$. For any $\beta, \mu \in \mathbb{R}$ and $\varepsilon >0,$ one has 
$$
\mathbf{I}(\zeta_T\in (\beta )_\varepsilon )\le e^{ \mu T (\zeta_T -\beta) +|\mu|T\varepsilon }
 = e^{ -\mu T \beta +|\mu|T\varepsilon -\mu Z(0)}e^{\mu Z(T) }.
$$
Hence for any fixed  $\mu \in \operatorname{dom}A $ and $\varepsilon=\varepsilon_T = o (1)$ as  $T\to\infty$, by virtue of condition~$[\mathbf{A}]$\ we have 
\begin{align}
\notag
\mathbf{P} (\zeta_T \in (\beta)_\varepsilon \, | \, \mathcal{F}_0) &\le e^{- \mu T \beta + |\mu|T\varepsilon -\mu Z(0)} \mathbf{E}\bigl(e^{\mu Z(T)} \bigm| \mathcal{F}_0\bigr)
\\
\label{eq2.5}
& = \exp\{ -T( \mu \beta - A(\mu)+o(1))\}
\end{align}
uniformly over elementary events 
$$
\omega \in C_{0}(T):=\biggl\{\biggl|\frac{Z(0)}{T}\biggr|<\eta_T\biggr\}\quad \text{with } \ \eta_T=o(1).
$$

First consider the case where   $\beta\in \operatorname{dom}D $. If  $\beta\in (\operatorname{dom}D )$ then there exists a $\mu (\beta)\in \operatorname{dom}A$ such that $\mu(\beta) \beta - A(\mu(\beta))=D( \beta)$. It immediately follows from~\eqref{eq2.5} with  $\mu = \mu (\beta)$ that 
\begin{equation}
\label{eq2.6}
\mathbf{P} (\zeta_T \in (\beta )_\varepsilon \,| \, \mathcal{F}_0) \le \exp\{ -T( D(\beta ) +o(1))\}
\end{equation}
with the same uniformity property for the residual term~$o(1) $ as in~\eqref{eq2.5}. It follows directly from this that a uniform conditional version of a relation of the form~\eqref{eq2.2}  holds true.

Now let $\beta\in \partial\operatorname{dom}D$. Assume for the definiteness that   $\beta=\sup \operatorname{dom}D$, so that  $D(\beta )<\infty$, $D(\beta + )=\infty$. It follows from here, \eqref{eq1.2} and condition~[$\mathbf{A}$] that $\sup\operatorname{dom}A =\infty$ and $A'(\mu)\nearrow \beta$ as $\mu\to\infty$. Therefore, for any sequence $\mu_n\to\infty$ as $n\to\infty$ one has  $\mu_n \beta - A(\mu_n)\to D( \beta)$. Using the standard  ``diagonal argument'' (see e.g.\ \cite[\S\,7.16]{15}), we obtain from here and~\eqref{eq2.5} the uniform bound~\eqref{eq2.6}.

In the case where  $\beta\notin \operatorname{dom}D $, for any $N>0$ there exists a $\mu_N\in \operatorname{dom}A$ such that $\mu_N \beta - A(\mu_N)\ge N$. From~\eqref{eq2.5} with $\mu =\mu_N$ we obtain that 
$$
\mathbf{P} (\zeta_T \in (\beta )_\varepsilon \, | \, \mathcal{F}_0) \le \exp\{ -T(N +o(1))\}
$$
uniformly over elementary events  $\omega \in C_{0}(T) $, which establishes the desired bound due to the arbitrariness of~$N$.

(ii)~Now we will obtain the desired  \textit{lower bound\/} for the conditional probability on the left-hand side of~\eqref{eq2.5} when $\beta\in \operatorname{dom}D$ (note that, when $\beta\notin \operatorname{dom}D,$ there is no need for a lower bound). To this end we will use the idea from the proof of Theorem~2.3.6 in~\cite{5}, namely, we will make use of the change of measure.

Denote by~$P_T(\,{\cdot}\,):=\mathbf{P} (\zeta_T\in \,{\cdot}\, | \, \mathcal{F}_0)$ the conditional distribution of~$\zeta_T$ and, for a given $\mu\in ( \operatorname{dom}A ) $, introduce a new (random) distribution~$\widetilde{P}_{T, \mu}$ on~$\mathbb{R} $, specifying it via its density 
$$
\frac{d\widetilde{P}_{T,\mu}}{dP_T}(x) := \frac{e^{ \mu T x }}{ E_{T,\mu} }, \quad \text{where }
 \ E_{T,\mu}:=\mathbf{E}\bigl( e^{ \mu Z(T)} \bigm| \mathcal{F}_0\bigr).
$$
Then, for $\beta \in \mathbb{R}  $ and  $\varepsilon >0,$ one has
$$
P_T((\beta)_\varepsilon) =E_{T,\mu} \int_{(\beta)_\varepsilon}e^{- \mu T x} \, d\widetilde{P}_{T, \mu}(x),
$$
so that 
\begin{align}
\frac1T \ln P_T((\beta)_\varepsilon) &= \frac1T \ln E_{T,\mu} - \mu \beta + \frac1T\ln \int_{(\beta)_\varepsilon}e^{\mu T(\beta - x)} \, d\widetilde{P}_{T, \mu}(x)
\nonumber
\\
&\ge \frac1T \ln E_{T,\mu} - \mu \beta -|\mu| \varepsilon + \frac1T\ln \widetilde{P}_{T, \mu}((\beta)_\varepsilon).
\label{eq2.7}
\end{align}
Here, due to condition [$\mathbf{A}$], we have 
\begin{equation}
\label{eq2.8}
\frac1T \ln E_{T,\mu} - \mu \beta = - (\mu \beta - A(\mu)) + \Theta_\mu (T),
\end{equation}
where $\Theta_\mu (T) = o(1)$ uniformly over elementary events  $\omega \in C_{0}(T)$. In the remaining part of the proof, all the relations including conditional probabili\-ties or conditional expectations, will also has that uniformity property, which we will stop mentioning for brevity.  


Since  $\mu \beta - A(\mu)\le D(\beta)$, $\mu\in \mathbb{R}$, we have thus established that, with  $\varepsilon =\varepsilon_T =o(1)  $ as $T\to\infty, $ one has the inequality 
$$
\frac1T \ln P_T((\beta)_\varepsilon) \ge - D(\beta) +o(1) + \frac1T\ln \bigl[1- \widetilde{P}_{T, \mu}((\beta)_\varepsilon^{\mathrm{c}})\bigr].
$$
Now the lower bound will be obtained once we have shown that 
\begin{equation}
\label{eq2.9}
\widetilde{P}_{T, \mu}( (\beta)_\varepsilon^{\mathrm{c}})=o(1)\quad \text{for } \ \varepsilon=\varepsilon_T = \overline{o}(1).
\end{equation}

For a fixed $\lambda >0$ such that  $\mu + \lambda \in \operatorname{dom}A$, employing~\eqref{eq2.8} one gets
\begin{align}
\notag
&\widetilde{P}_{T, \mu}( [\beta+\varepsilon, \infty)) \le e^{- \lambda T(\beta +\varepsilon)} \int e^{\lambda Tx } \, \widetilde{P}_{T, \mu} (dx) = e^{- \lambda T(\beta +\varepsilon)} \, \frac{E_{T,\mu +\lambda}}{E_{T,\mu}}
\\
&\qquad= \exp\bigl\{T[ A(\mu+\lambda) - A(\mu) - \lambda (\beta +\varepsilon ) +\Theta_{\mu+\lambda} (T) - \Theta_\mu (T)]\bigr\}.
\label{eq2.10}
\end{align}

{ If $\beta\,{\in}\, (\operatorname{dom}D)$ then, by virtue of condition~[$\mathbf{A}$], there exists a point  $\mu\,{=}\,\mu (\beta)\,{\in}\,(\operatorname{dom}D)$} 
such that $A'(\mu)\,{=}\,\beta$ 
(recall that a differentiable convex function has a continuous derivative; see also \cite[Theorem~26.5]{12}). Therefore, with such a $\mu$, one has the equality 
$$
A(\mu +\lambda) = A(\mu) + \beta\lambda + a(\lambda)\lambda, \qquad \lambda>0,
$$
where $\overline{a}(h):= \sup_{0\le \lambda\le h}|a(\lambda)|=o(1)$ as $h\to 0$, so that the argument of the exponential in~\eqref{eq2.10} takes the form 
$$
-T\bigl(\lambda \varepsilon - \Theta_{\mu+\lambda} (T) + \Theta_{\mu} (T) + a(\lambda)\lambda\bigr).
$$
We will now show that one can choose $\varepsilon= \overline{o} (1)$ and $\lambda= \overline{o} (1)$ as $T\to\infty$ such that the above expression will tend to~$-\infty$.

Choose a   $\lambda^{(0)}>0$ such that $\mu + \lambda^{(0)}\in (\operatorname{dom}A)$, and set  $\lambda^{(k)}:= \lambda^{(0)}/k$, $k=1,2,\dots$\,. Further, one can always find   values $0=T_0<T_1 <T_2 <\cdots$ such that 
$$
\sup_{T\ge T_k} (|\Theta_{\mu+\lambda^{(k)}}(T)| + |\Theta_{\mu}(T)|) \le \overline a(\lambda^{(k)})\lambda^{(k)}, \qquad k\ge 1.
$$
Setting 
\begin{align*}
\lambda &=\lambda_T:= \sum_{k\ge 0} \lambda^{(k)} \, \mathbf{I}(T_k \le T<T_{k+1}),
\\
\varepsilon &=\varepsilon_T:= \sum_{k\ge 0} \overline a(\lambda^{(k)})^{1/2} \, \mathbf{I}(T_k \le T<T_{k+1}),
\end{align*}
we complete the desired construction.

Using $\lambda <0$, we obtain  in the same way that $\widetilde{P}_{T, \mu}( (- \infty, \beta - \varepsilon]) \to 0$. This establishes relation~\eqref{eq2.9} and hence the desired lower bound for  $P_T((\beta)_\varepsilon) $ in the case where $\beta\in (\operatorname{dom}D)$.

It remains to obtain the lower bound in the case where   $\beta\in \partial \operatorname{dom}D$. Assume for definiteness that  $\beta=\sup \operatorname{dom}D$. Fix an $\varepsilon >0$ such that 
$$
\beta':= \beta - \varepsilon' \in (\operatorname{dom}D) , \quad \text{where }\  \varepsilon' := \frac{\varepsilon}2.
$$
Using~\eqref{eq2.7} and~\eqref{eq2.8}, it is easily seen that 
\begin{align*}
\frac1T \ln  P_T((\beta)_\varepsilon) &\ge \frac1T \ln  P_T((\beta')_{\varepsilon'})
\\
&\ge \frac1T  \ln E_{T,\mu} - \mu \beta' -|\mu| \varepsilon' + \frac1T\ln \widetilde{P}_{T,\mu}((\beta')_{\varepsilon'})
\\
&\ge - D(\beta) - |\mu| \varepsilon + \Theta_\mu (T) + \frac1T\ln  \widetilde{P}_{T, \mu}((\beta')_{\varepsilon'}).
\end{align*}
Since $\beta' \in (\operatorname{dom}D)$, the above argument given for the case where  $\beta\in (\operatorname{dom}D) $, in which we now chose  $\mu = \mu (\beta')$ and a sufficiently small $\lambda$, shows that 
$$
\frac1T \ln  P_T((\beta)_\varepsilon) \ge -D(\beta) - 2 |\mu| \varepsilon
$$
for all large enough~$T$. As $\varepsilon >0$ is arbitrary, this implies the desired lower bound. Theorem~\ref{t2.2} is proved.
}
\end{proof}


\section{Functional large deviation principle}
 \label{s3}

\subsection{Conditional large deviation principle for finite-dimensional distributions}
\label{ss3.1}
First we will extend the assertion of Theorem~\ref{t2.2} to the case of finite-dimensional distributions of the process~$Z$. For given integer $K > 1$ and reals 
$$
h_1, \dots, h_K >0 \quad \text{such that }\ \sum_{j=1}^K h_j=1,
$$
introduce the partition $\boldsymbol{s}^K:=\{s_0, s_1, \dots, s_K\}$  of the interval $[0,1]$, where 
\begin{equation}
\label{eq3.1}
s_0:=0, \qquad s_k:= \sum_{j=1}^k h_j, \quad k\in [K ]:=\{1,\dots, K \},
\end{equation}
and set  $T_k: =s_k T$, $0\le k \le K$.

Further, for given $ \alpha_0, \alpha_1, \dots, \alpha_K \in \mathbb{R},$ introduce the following notations: 
\begin{gather}
\beta_k := \frac{\alpha_k-\alpha_{k-1}}{h_k}, \qquad \zeta_T^{(k)}:= \frac{ Z(T_k)- Z(T_{k-1})}{T_k - T_{k-1}},
\label{eq3.2}
\\
G_{k,T}( \varepsilon) := \bigl\{\zeta_T^{(k)} \in (\beta_k)_\varepsilon\bigr\}, \qquad k\in [K ], \quad \varepsilon >0.
\notag
\end{gather}

\begin{theorem}
\label{t3.1}
Let condition  $[\mathbf{A}]$ be met, $ D(\alpha): =\mathcal L_A(\alpha)$, $\alpha\in \mathbb{R}$, and a positive function $\eta=\eta_T=o(1)$ as $T\to\infty$ be given. Then, for $\varepsilon_T = \overline{o} (1)$, 
the   relations 
\begin{equation}
\lim_{T\to 0}\frac1T \ln \mathbf{P} \biggl(\bigcap_{k=1}^K G_{k,T} (\varepsilon_T) \biggm|\mathcal{F}_0 \biggr) = - \sum_{k=1}^K h_k D (\beta_k) 
\label{eq3.3}
\end{equation}
hold true uniformly over elementary events 
$$
\omega \in C_{0}(T):= \biggl\{\biggl|\frac{Z(0)}{T} -\alpha_0\biggr|<\eta_T\biggr\}.
$$
\end{theorem}

\begin{proof}
{\rm 
First we will derive the upper bound under the assumption that 
\begin{equation}
\label{eq3.4}
\beta_k\in (\operatorname{dom}D),\qquad k\in [K].
\end{equation}
For $\varepsilon >0$ one has 
\begin{align}
\mathbf{P} \biggl(\bigcap_{k=1}^K G_{k,T} (\varepsilon ) \biggm| \mathcal{F}_0 \biggr) &= \mathbf{E} \biggl[ \prod_{k=1}^K \mathbf{I} (G_{k,T} (\varepsilon ))\biggm| \mathcal{F}_0 \biggr]
\notag
\\
&=\mathbf{E} \biggl[ \mathbf{E} \bigl[\mathbf{I} (G_{K,T} (\varepsilon )) \bigm| \mathcal{F}_{T_{K -1}}\bigr] \prod_{k=1}^{K -1} \mathbf{I} (G_{k,T} (\varepsilon )) \biggm| \mathcal{F}_0 \biggr].
\label{eq3.5}
\end{align}
Since
$$
Z(T_{K -1}) = \sum_{k=1}^{K-1} (Z(T_k)- Z(T_{k-1})) + Z(0) = T \sum_{k=1}^{K-1} h_k\zeta_T^{(k)} + Z(0),
$$
we see that, for a positive function   $\varepsilon_T =o(1) $, on the event 
$$
B_{K -1}:= C_{0}(T) \bigcap_{k=1}^{K -1} G_{k,T} (\varepsilon_T)
$$
one has the inclusion 
\begin{equation}
\label{eq3.6}
\frac{Z(T_{K -1})}{T}\in (\alpha_{K -1})_{\varepsilon_T^{(k)}}, \quad \text{where } \ \varepsilon_T^{(k)}:= \eta_T + \sum_{k=1}^{K -1} h_k \varepsilon_T \le \eta_T + \varepsilon_T \to 0.
\end{equation}
Therefore it follows from part~(i) of the proof of Theorem~\ref{t2.2} that 
\begin{align*}
\mathbf{E} [\mathbf{I} (G_K (\varepsilon_T)) \mid \mathcal{F}_{T_{K -1}}] &= \mathbf{P} \bigl(\zeta_T^{(K )}\in (\beta_K )_{\varepsilon_T} \bigm| \mathcal{F}_{T_{K -1}}\bigr)
\\
&\le \exp\bigl\{ -T\bigl( D(\beta_K ) +o(1)\bigr)\bigr\}
\end{align*}
uniformly over elementary events  $\omega \in B_{K -1}$. Hence the conditional probability on the left-hand side of~\eqref{eq3.5} is equal to 
$$
\exp\bigl\{ -T\bigl( D(\beta_K ) +o(1)\bigr)\bigr\} \, \mathbf{E} \biggl[ \prod_{k=1}^{K -1} \mathbf{I} (G_{k,T} (\varepsilon_T)) \biggm| \mathcal{F}_0 \biggr],
$$
where the residual term $o(1)$ is uniform over elementary events $\omega \in C_{0}(T)$. 
Repeating this procedure $K -2$ times, we obtain the upper bound needed to establish relation~\eqref{eq3.3} in the case where~\eqref{eq3.4} holds.

The case where not all of the relations in~\eqref{eq3.4} hold true is considered in a similar way, using the respective arguments from part~(i) of the proof of Theorem~\ref{t2.2}. 

Now we will obtain the lower bound for the conditional probability on the left-hand side of~\eqref{eq3.3}. Clearly, it is sufficient to consider the case where  $\beta_k\in \operatorname{dom}D, $ $k\in [K]$.  

It follows from part~(ii) of the proof of Theorem~\ref{t2.2} that there exists $\widetilde \varepsilon_T^{\,(1)}= o(1)$ as $T\to \infty$ such that the relation 
$$
\mathbf{P} \bigl(G_{1,T} (\widetilde \varepsilon_T^{\,(1)}) \bigm| \mathcal{F}_0\bigr) \ge \exp\{- h_1 T (D(\beta_1)+o(1))\}
$$
holds uniformly over elementary events  $\omega \in C_{0}( T) $. Similarly, there exists  $\widetilde \varepsilon_T^{\,(2)}=o(1)$, $T\to \infty$, such that the relation 
$$
\mathbf{P} \bigl(G_{2,T} (\widetilde \varepsilon_T^{\,(2)}) \bigm|\mathcal{F}_{T_1}\bigr) \ge \exp\{- h_2 T (D(\beta_2)+o(1))\}
$$
holds  uniformly over elementary events  $ \omega\in \{ |Z(T_1) /T-\alpha_1|<\eta^{(1)}_T\} $, where $ \eta_T^{(1)}:= \eta_T + h_1 \widetilde \varepsilon_T^{\,(1)}$, and so on. 

Putting $\widetilde \varepsilon_T:=\max_{1\le k\le K }\widetilde \varepsilon_T^{\,(k)}$, for the probability on the left-hand side of~\eqref{eq3.3} one has 
\begin{align*}
&\mathbf{P} \biggl(\bigcap_{k=1}^K G_{k,T} (\widetilde \varepsilon_T) \biggm|\mathcal{F}_0 \biggr)
\\
&\qquad \ge \mathbf{P} \biggl(\bigcap_{k=1}^K G_{k,T} (\widetilde \varepsilon_T^{\,(k)}) \biggm|\mathcal{F}_0 \biggr) =\mathbf{E} \biggl[\prod_{k=1}^K \mathbf{I} \bigl(G_{k,T} (\widetilde \varepsilon_T^{\,(k)})\bigr) \biggm|\mathcal{F}_0 \biggr]
\\
&\qquad = \mathbf{E} \biggl[ \mathbf{E} \bigl[ \mathbf{I} \bigl(G_{K,T}(\widetilde \varepsilon_T^{\,(K)})\bigr) \bigm| \mathcal{F}_{T_{K -1}}\bigr] \prod_{k=1}^{K -1} \mathbf{I} \bigl(G_{k,T} (\widetilde \varepsilon_T^{\,(k)})\bigr) \biggm|\mathcal{F}_0 \biggr]
\\
&\qquad\ge \mathbf{E} \biggr[ \exp\{- h_K T (D(\beta_K )+o(1))\} \prod_{k=1}^{K -1} \mathbf{I} \bigl(G_{k,T} (\widetilde \varepsilon_T^{\,(k)})\bigr)\biggm|\mathcal{F}_0 \biggr]
\\
&\qquad = \exp\{- h_K T (D(\beta_K )+o(1))\} \, \mathbf{E} \biggl[ \prod_{k=1}^{K -1} \mathbf{I} \bigl(G_{k,T} (\widetilde \varepsilon_T^{\,(k)})\bigr)\biggm|\mathcal{F}_0 \biggr],
\end{align*}
where the residual term  $o(1) $ is uniform over  $\omega \in \{ |Z(T_{K-1}) /T-\alpha_{K-1}|<\eta^{(K-1)}_T\} $.

Repeating this ``splitting procedure'' with the remaining conditional ex\-pec\-tation further~$K -2$ times (the desired uniformity of the residual terms is guaranteed by the choice of   $\widetilde \varepsilon_T^{\,(k)}$), we obtain 
$$
\ln \mathbf{P} \biggl(\bigcap_{k=1}^K  G_{k,T} (\widetilde \varepsilon_T ) \biggm|\mathcal{F}_0 \biggr) \ge -T \sum_{k=1}^K h_k D (\beta_k)+o(T).
$$
This bound will all the more be valid if we replace  $\widetilde \varepsilon_T $ by $\varepsilon_T^* \gg \widetilde \varepsilon_T, $ $ \varepsilon_T^* =o(1)$. Therefore it will also hold when we replace  $\widetilde \varepsilon_T $ with $\varepsilon_T=\overline{o}(1)$. Theorem~\ref{t3.1} is proved.
}
\end{proof}

\subsection{Functional large deviation principle}
\label{ss3.2}

In this section, we assume that condition~[$\mathbf{A}$] is met and set $ D(\alpha): =\mathcal L_{ A} (\alpha)$, $\alpha\in \mathbb{R}$.

By assumption, the trajectories of the process   $ \{z_T(s) = Z(sT)/T\colon s\in [0,1]\}$ belong to the space $\mathbb{D}=\mathbb{D}[0,1]$ of right-continuous functions without discontinuities of the second kind on $[0,1]$. Denote  by 
$$
\|f\|:=\sup_{s\in [0,1]}|f(s)|,\qquad f\in \mathbb D,
$$
the uniform norm on that space and by 
$$
(f)_\varepsilon:= \{g\in \mathbb D \colon \|g-f\|<\varepsilon \}, \qquad \varepsilon >0,
$$
the open $\varepsilon$"=neighborhood of~$f$ in that norm. In this section, we will make use of Theorem~\ref{t3.1} to obtain the conditional l.l.d.p.\ for the  trajectories of~$z_T$. To this end, we will need several concepts from~\cite{4}, \cite{16}.

Denote by $\mathbb C= \mathbb C [0,1]$ the space of continuous functions on $[0,1]$ and by $\mathbb C_a\subset \mathbb C$ the class of absolutely continuous functions on that interval. For  $f\in \mathbb C_a$ there always exists the Lebesgue integral 
$$
I (f) := \int_{[0,1]} D(f'(s)) \, ds \le \infty.
$$

Further, for a function  $f \in \mathbb D$, we will say  the deviation integral  $J(f)$ exists provided that, for any sequence of partitions $\boldsymbol{s}^{K}$ of the form~\eqref{eq3.1} such that $\max_{k\le K} h_k\to 0$ as $K\to \infty$, there exists the limit 
$$
J(f):= \lim_{K\to \infty}\sum_{k=1}^K h_k D \biggl(\frac{f(s_k)-f(s_{k-1})}{h_k}\biggr) 
$$
that does not depend on the choice of the sequence of partitions~$\boldsymbol{s}^{K}$. The functional $J(f)$ is nothing but the Darboux integral of the interval function  
$$
F(s,t):= (t-s) D \biggl(\frac{f(t) - f(s)}{t-s}\biggr),\qquad 0\le s <t\le 1
$$
(see e.g.\ \cite[Ch.~1]{17}).   Note  also that, in the present exposition, it will be convenient for us  to drop the convention used in~\cite{4}, \cite{16} that $J(f)=\infty$  whenever $f(0)\neq 0$. 

It was proved in Theorem~2.1 in~\cite{16} that the deviation integral  $J(f)$ (finite or infinite) always exists for  $f\in \mathbb D$ and, moreover, that 
\begin{equation}
\label{eq3.7}
J(f) = \sup I (f^{\boldsymbol{s}^{K}}),
\end{equation}
where 
$$
f^{\boldsymbol{s}^K}(s) :=\frac{s_k- s}{h_k} \, f(s_{k-1}) + \frac{s- s_{k-1}}{h_k} \, f(s_k), \qquad s\in [s_{k-1}, s_k], \quad k\in [K ],
$$
is a piece-wise linear function with nodes at the points  $(s_k, f(s_k))$, $0\le k\le K$, and the supremum is taken over all finite partitions $\boldsymbol{s}^{K}$ of the interval $[0, 1]$.

\begin{remark}
\label{r3.1}
{\rm  Strictly speaking, the concept of the deviation integral was introduced in~\cite{16} in the case where~$D$ is a rate function, i.e.\ the Legendre transform of the function  $\ln \mathbf{E} e^{\mu \xi}$, $\mu \in \mathbb{R},$ for a given random variable~$\xi$. However, it is not hard to see that our function  $D=\mathcal L_A$, where $A$ is the fundamental function from condition~[$\mathbf{A}$], possesses all the properties of the rate functions that were used when studying the deviation integral in~\cite{16}.
}
\end{remark}

Since the process $Z$ is of general form,  getting results on the asymptotic behavior of  $\ln \mathbf{P} (z_T\in (f)_{\varepsilon_T})$ as $T\to\infty $ without additional assumptions enabling one to ``control''  the oscillations of the trajectory of~$z_T$ on short time intervals appears impossible. We will introduce the following simple condition. 

\smallskip 

[$\mathbf{B}$].~{\it There exist positive functions $V(t) = o(t)$ as $ t\to\infty$ and $W(h) = o(1)$ as $h \to 0$ such that, for any $T>0$ and $\delta >0,$ the following inequality  holds a.s.$:$}
$$
\sup \bigl\{|Z(uT) - Z(vT)|\colon 0\le u < v \le \min (u+\delta, 1) \bigr\} \le V(T) + W(\delta) T.
$$

\smallskip 

It is easily seen that condition  [$\mathbf{B}$] is met, in particular, for processes with ``almost Lipschitz'' trajectories, when there exist constants   $\gamma_0, \gamma_1, v_0 \ge 0$ such that 
$$
\sup_{0\le u\le v}|Z(t+u) - Z(t)| \le \gamma_0 + \gamma_1 v \quad \text{a.s. \ \ for any } \ t\ge 0,\ \ v\ge v_0.
$$
This latter condition is satisfied, for example, for  the partial sums process $Z(t)=Z(0) +\sum_{j=1}^{\lfloor t\rfloor} \xi_j$, $t\ge 0$, where $\{\xi_j\colon j\ge 1\}$  is a sequence of bounded random variables: $ \mathbf{P} (|\xi_j | \le \gamma )=1$ for some  $\gamma >0$ and all $ j\ge 1$ (in this case, one can take $\gamma_0=\gamma_1=\gamma$, $v_0=0$).

It is clear that if both conditions~[$\mathbf{A}$] and~[$\mathbf{B}$] are met then $\operatorname{dom}A = \mathbb{R}$, which is the exact analogue of condition   $[\mathbf C_\infty]$ from~\cite{16} in what concerns the properties of the deviation integral for the function~$D$. Therefore, according to Theorem~2.2\,(iv) in~\cite{16}, in this case one has 
\begin{equation}
\label{eq3.8}
J(f)=\infty \quad\text{for  } \ f\in \mathbb D\setminus \mathbb C_a.
\end{equation}
Moreover, it was established in Theorem~2.2\,(v) from the same paper that   $J(f)= I (f)$ for $f\in \mathbb C_a$.

\begin{theorem}
\label{t3.2}
{\rm(i)}~Let condition $[\mathbf{A}]$ be satisfied for the process~$Z$. Then, for any $f\in \mathbb D$,  for  $\eta_T=o(1)$, $\varepsilon_T=\overline{o} (1)$ the inequality 
$$
\limsup_{T\to \infty} \frac1T \ln \mathbf{P} \bigl( z_T \in (f)_{\varepsilon_T} \bigm| \mathcal{F}_{0}\bigr) \le -J(f) 
$$
holds uniformly over elementary events $\omega \in \{|Z(0)/T - f(0)| < \eta_T \}.$ 

{\rm (ii)}~If, moreover, condition~$[\mathbf{B}]$ is also met  
then the equality 
$$
\lim_{T\to \infty} \frac1T \ln \mathbf{P} \bigl( z_T \in ( f)_{\varepsilon_T} \bigm| \mathcal{F}_{0}\bigr) = -J(f), \qquad f\in \mathbb D,
$$
holds true with the same uniformity property. 
\end{theorem}

\begin{remark}
{\rm
As we already mentioned in the Introduction, the assertions of Theorems~\ref{t2.2}, \ref{t3.1} and \ref{t3.2} will remain true in a more general triangular array scheme as well, where the process $Z=Z^{(T)}$ itself can also depend on parameter~$T$. In such a setup, condition~[$\mathbf{B}$] looks more natural. Moreover, it can be extended to formulations involving suitable upper bounds for the conditional probabilities of large values of the oscillations of the trajectories of~$Z$ instead of a.s.\ restrictions on their behavior. We do not present possible   variants of such conditions because of their complexity, in order to avoid complicating the exposition.
}
\end{remark}

\begin{proof}\!\!~{\bf of Theorem~\ref{t3.2}.}
{\rm 
(i)~\textit{Upper bound}.~For a given funtion $f\in \mathbb D$, integer number  $K > 1 $ and partition $\boldsymbol{s}^K$ of the form~\eqref{eq3.1}, we will put 
$$
\alpha_0:= f(0), \qquad \alpha_{k}:= f(s_k ), \quad k\in [K],
$$
and make use of the notations from~\eqref{eq3.2}. Clearly, for any   $\varepsilon >0$, $k \in [K],$ one has the inclusions 
$$
\{z_T\in (f)_{\varepsilon }\} \subseteq \biggl\{ \frac{Z(T_k)-Z(T_{k-1})}{T} \in (\alpha_k -\alpha_{k-1})_{2\varepsilon} \biggr\} \subseteq \bigl\{ \zeta^{(k)}_T \in (\beta_k)_{\varepsilon^*} \bigr\},
$$
where $\varepsilon^*:= 2\varepsilon \max_{k\in [K]}h^{-1}_k$. From this it follows that 
$$
\mathbf{P} \bigl(z_T\in (f)_{\varepsilon } \bigm|\mathcal{F}_0 \bigr) \le \mathbf{P} \biggl(\bigcap_{k=1}^K G_{k,T} (\varepsilon^*) \biggm|\mathcal{F}_0 \biggr).
$$
Hence, if $\eta_T=o(1)$ and $\varepsilon_T = \overline{o} (1)$ as $T\to\infty$ then, by Theorem~\ref{t3.1}, one has 
$$
\limsup_{T\to 0}\frac1T \ln \mathbf{P} \bigl(z_T\in (f)_{\varepsilon_T } \bigm|\mathcal{F}_0 \bigr) \le -\sum_{k=1}^K h_k D (\beta_k) = - I (f^{\boldsymbol{s}^K})
$$
uniformly over elementary events  $\omega\in\{ |z_T(0)-\alpha_0|<\eta_T\}$. Therefore, by virtue of~\eqref{eq3.7}, we will also have the inequality 
$$
\limsup_{T\to 0}\frac1T \ln \mathbf{P} \bigl(z_T\in (f)_{\varepsilon_T } \bigm|\mathcal{F}_0 \bigr) \le -J(f).
$$

(ii)~\textit{Lower bound}. It follows from~\eqref{eq3.8} that we only need to consider the case where $f\in \mathbb C_a$. Put $h_k:=K^{-1}$, $k\in [K]$. By virtue of condition~[$\mathbf{B}$], for $s\in [s_{k-1}, s_k]$ one has 
\begin{equation}
\label{eq3.9}
z_T (s) - f(s) \le z_T (s_{k-1}) +V(T) T^{-1} + W( K^{-1}) - f(s_{k-1}) + \omega_f (K^{-1}),
\end{equation}
where $\omega_f (\delta ):= \sup \{ |f(s) - f(t)|\colon 0\le s < t \le \min(s+\delta, 1) \}$, $\delta >0,$ is the continuity modulus of the function~$f$.

For any given $\varepsilon>0,$ choose  $K$ so large that $W( K^{-1}) + \omega_f (K^{-1}) < \varepsilon$. Then, for  $\eta <\varepsilon$ and $T$ so large that $W(T)T^{-1}< \varepsilon$, we obtain from~\eqref{eq3.9} and a similar lower bound for the difference  $z_T (s) - f(s)$ that 
$$
\{ |z_T(0)-\alpha_0|<\eta\}\bigcap_{k=1}^K G_{k,T} (\varepsilon )\subseteq \{z_T\in (f)_{3\varepsilon}\}
$$
(cf.~\eqref{eq3.6}). Therefore on the event $\{ |z_T(0)-\alpha_0|<\eta\}\in \mathcal{F}_0$ one has the inequality 
$$
\mathbf{P} \bigl(z_T\in (f)_{3\varepsilon} \bigm|\mathcal{F}_0 \bigr) \ge \mathbf{P} \biggl(\bigcap_{k=1}^K G_{k,T} (\varepsilon )\biggm| \mathcal{F}_0 \biggr).
$$
Hence,   for  $\eta_T =o(1)$, $\varepsilon_T =\overline{o}(1)$, by Theorem~\ref{t3.1} and by virtue of~\eqref{eq3.7}, the bound   
\begin{align*}
\liminf_{T\to\infty } \frac1T\ln \mathbf{P} \bigl(z_T\in (f)_{ 3 \varepsilon} \bigm|\mathcal{F}_0 \bigr) &\ge \liminf_{T\to\infty } \frac1T\ln \mathbf{P} \biggl(\bigcap_{k=1}^K G_{k,T} (\varepsilon_T) \biggm| \mathcal{F}_0 \biggr)
\\
& = - \sum_{k=1}^K h_k D (\beta_k) = - I \bigl(f^{\boldsymbol{s}^K}\bigr) \ge - J(f)
\end{align*}
will hold uniformly over elementary events  $\omega \in \{ |z_T(0)-\alpha_0|<\eta_T\}.$  
Since  $\varepsilon >0$ is arbitrary, Theorem~\ref{t3.2} is proved.
}
\end{proof}



\begin{thebibliography}{99}

\bibitem{1}
{\sc Cram\'er, H.} 
Sur un nouveau th\'eor\`eme–limite de la th\'eorie des probabilit\'es. In: {\em Actualit\'es Scientifiques et Industrielles}, number 736 in Colloque consacr\'e \`a la th\'eorie des probabilit\'es, pp. 5--23. Hermann, Paris, 1938. [English translation: {\tt arXiv:1802.05988v4}.]

\bibitem{2}
{\sc Borovkov, A.\,A., Mogulskii, A.\, A.}
{Large deviations and testing of statistical hypotheses. I. Large deviations of sums of random vectors.}
{\em Siberian Adv. Math.} 1992, v. 2, 52--120. 

\bibitem{3}
{\sc Borovkov, A.\, A.}
{Asymptotic Analysis of Random Walks: Light-Tailed Distributions.}
{Cambridge: Cambridge Univ.\ Press, 2020.}

\bibitem{4}
{\sc Borovkov, A.\,A., Mogulskii, A.\, A.} 
On large deviation principles for random walk trajectories.~I
{\em  Theory Probab.\ Appl,} 2011,   56, 538--561.

\bibitem{5}
{\sc Dembo\,A., Zeitouni. O.} 
Large Deviations Techniques and Applications. Corr.\ printing of 2nd edn., Springer, 2009.

\bibitem{6}
{\sc Deuschel, J.\,D., and  Stroock, D.\,W.}
{Large Deviations.} Academic Press, Boston, 1989.


\bibitem{7}
{\sc Puhalskii, A.}
{Large Deviations and Idempotent Probability.}
{Boca Raton, FL: CRC Press, 2001.}


\bibitem{8}
{\sc Varadhan, S. R. S.} 
Asymptotic probabilities and differential equations.
{\em Comm. Pure Appl. Math.} 1966, 19, 261--286. 




\bibitem{9}
{\sc Borovkov, A.\,A.} 
Compound renewal processes. 
{Cambridge: Cambridge Univ.\ Press, 2020.}


\bibitem{10}
{\sc G\"artner, J.} 
On large deviations from the invariant measure. 
{\em Theor.\ Prob.\ Appl.} 1977, 22, 24--39.

\bibitem{11}
{\sc Ellis, R. S.} 
Large deviations for a general class of random vectors.
{\em Ann.\ Probab.} 1984, 12, 1--12.


\bibitem{12}
{\sc Rockafellar~R.T.}
{Convex Analysis.} 
{Princeton: Princeton Univ.\ Press, 1970. }

\bibitem{13}
{\sc Borovkov, A.\,A., Mogulskii, A.\, A.} 
On large deviation principles in metric spaces
{\em Siberian Math.~J.}  2010,  51, 989--1003.

\bibitem{14}
{\sc  Nguyen~D.\,P.,   Borovkov~K.}
Parisian ruin with random deficit-dependent delays for spectrally negative Lévy processes.
{\em Insur.\ Math.\ Econ.} 
2023, 110,   72--81.

\bibitem{15}
{\sc Royden~H.L., Fitzpatrick~P.M.} 
{Real Analysis.} 4th ed. 
{Boston: Pearson,  2010.}

\bibitem{16}
{\sc Borovkov,  A.~A.,  Mogul'ski,  A.~A.} 
Properties of a~functional of trajectories which arises in studying the probabilities of large deviations of random walks. 
{\em Siberian Math.~J.}  2011,  52, 
612--627.

\bibitem{17}
{\sc Riesz, F., Sz.-Nagy, B.}
Functional Analysis. 
{New York: Dover, 1990.}
 
 
\end{thebibliography}
\end{document}